\documentclass[11pt]{article}

\usepackage[T1]{fontenc}
\usepackage[margin=1in]{geometry}
\usepackage{amsmath,amssymb,amsthm,mathtools}
\usepackage{microtype}
\usepackage[hidelinks]{hyperref}

\newtheorem{theorem}{Theorem}
\newtheorem{proposition}[theorem]{Proposition}
\newtheorem{lemma}[theorem]{Lemma}
\newtheorem{corollary}[theorem]{Corollary}

\theoremstyle{remark}
\newtheorem{remark}[theorem]{Remark}
\newcommand{\R}{\mathbb R}
\newcommand{\norm}[1]{\left\lVert #1\right\rVert}
\newcommand{\abs}[1]{\left\lvert #1\right\rvert}
\newcommand{\dd}{\,\mathrm d}
\newcommand{\D}{\mathcal D}

\title{Discontinuity of the Vlasov--Poisson Flow
in $L_x^pL_v^\infty$}
\author{Ke Chen\and In-Jee Jeong \and Quoc-Hung Nguyen \and Sangwook Tae}
\date{}

\begin{document}
\maketitle

\begin{abstract}
We prove that solution maps of the Vlasov--Poisson equation are discontinuous at the zero initial datum in $X_p=L_x^pL_v^\infty$ for every $1\leq p<\infty$, in $\mathbb{R}^{d} \times \mathbb{R}^{d}$ with $d \le 3$, and for both
attractive and repulsive interactions. For every $T>0$, the trajectory of the solution in $L^\infty([0,T];X_p)$ is
discontinuous at zero, and for every sufficiently small fixed $t>0$, the fixed-time map with values in $X_p$ is also discontinuous at zero. 
The counterexamples are smooth, nonnegative, bounded by one, and supported
in a common compact subset of phase space.  Their mass and initial
$X_p$ norm tend to zero, whereas their solution norm at the observation
time is at least one.  The construction places narrow spatial packets
on a lattice.  Free transport allows a different velocity to select a
packet at each spatial point, while a smallness estimate on the field ensures that the nonlinear characteristics are close to that of the free transport. 
For the same families, the phase-space $L^q$ norms converge to zero uniformly in time for every $1\leq q<\infty$.
\end{abstract}


\section{Introduction}\label{sec:introduction}

The Vlasov--Poisson equation describes the transport of a particle
distribution by the force generated by its own density.  We shall 
consider the Cauchy problem on $\R_x^d\times\R_v^d$, with $d \le 3$,
\begin{equation}\label{eq:VP}
 \partial_t f+v\cdot\nabla_x f+\sigma E_f\cdot\nabla_v f=0,
 \qquad
 E_f=K_d*\rho_f,
 \qquad
 \rho_f(t,x)=\int_{\R^d}f(t,x,v)\dd v,
\end{equation}
for $t\geq0$, with $f(0,x,v)=f_0(x,v)$ and
\begin{equation}\label{eq:kernel}
 K_d(x)=c_d\frac{x}{\abs{x}^{d}},
 \qquad c_d=\abs{\mathbb S^{d-1}}^{-1},
 \qquad \sigma\in\{-1,1\}.
\end{equation}
Here convolution is in the spatial variable, and $\mathbb S^{d-1}$ is
the unit sphere. The formula \eqref{eq:kernel} also works in the one dimensional case $d = 1$, with $K_{1}(x) = \frac12 \mathrm{sgn}(x)$. With this normalization, $\sigma=1$ gives repulsive
interactions and $\sigma=-1$ gives attractive interactions.  The estimates
below are uniform in this choice of sign.

In this paper, we are only concerned with smooth and compactly supported data, and the existence and uniqueness of global classical solutions are well-known. In one and two dimensions this follows from Ukai and
Okabe~\cite{UkaiOkabe}; in three dimensions it was obtained by
Pfaffelmoser~\cite{Pfaffelmoser}.  See also
Horst~\cite{Horst,Horst2} for the classical Cauchy theory.

We ask whether these classical solutions depend continuously on their
initial data when both are measured in the mixed space
\begin{equation}\label{eq:Xp}
 X_p=L_x^pL_v^\infty,
 \qquad
 \norm{g}_{X_p}
 =\left(\int_{\R^d}
     \left(\mathop{\rm ess\,sup}_{v\in\R^d}\abs{g(x,v)}\right)^p
 \dd x\right)^{1/p}.
\end{equation}
Here $1\leq p<\infty$.  The characteristic flow preserves phase-space
volume and all phase-space $L^q$ norms.  Kinetic transport does not,
however, preserve $X_p$: the spatial displacement depends on velocity,
whereas the velocity supremum in \eqref{eq:Xp} is taken separately at
each spatial point.  Already for free transport, denoted by $U(t)$,
\[
 U(t)g(x,v) := g(x-tv,v),
\]
the maximizing velocity may depend on $x$.  A spatially sparse datum can
therefore have small $X_p$ norm initially while its transported velocity
envelope fills a set of fixed spatial measure.  The question is whether
the nonlinear term in the Vlasov--Poisson field destroys this mechanism.

%All Lebesgue spaces use Lebesgue measure.  

%The dependence of constants is indicated in each statement.  In the packet estimates, constants may depend on $d$ and the fixed cutoffs, but not on $h$, $r$, or the sequence index; constants in the mixed-norm bounds may also depend on $p$.

Fix $d \le 3$ and $\sigma\in\{-1,1\}$, and let
\begin{equation}\label{eq:data-class}
 \D=\{g\in C_c^\infty(\R^{2d}):g\geq0\}.
\end{equation}
For each $f_0\in\D$, let $f$ be the corresponding global classical
solution.  It solves \eqref{eq:VP} pointwise and is transported by its
characteristic flow:
\begin{equation}\label{eq:characteristics}
 \dot X(s)=V(s),\qquad
 \dot V(s)=\sigma E_f(s,X(s)).
\end{equation}
We denote the characteristic ending at $(x,v)$ at time $t$ by
$(X(s;t,x,v),V(s;t,x,v))$.  It satisfies
\begin{equation}\label{eq:transport}
 f(t,x,v)=f_0\bigl(X(0;t,x,v),V(0;t,x,v)\bigr).
\end{equation}
On each finite time interval the phase-space support is contained in
a compact set, since it is transported by a continuous flow.

For $t\geq0$ and $T>0$, define the maps
\[
 S_t:\D\to X_p,\quad S_tf_0=f(t),\qquad
 \mathcal S_T:\D\to L^\infty([0,T];X_p),\quad
 \mathcal S_Tf_0=f|_{[0,T]},
\]
where we endow $\D$ with the topology induced by $X_p$.
%Both maps are defined on classical data.

Each solution belongs to $C([0,T];X_p)$: to see this, let $K_x$ be a compact set containing
the spatial support on $[0,T]$.  Uniform continuity on a compact set
containing the transported phase-space support gives
\[
 \norm{f(s)-f(t)}_{X_p}
 \leq \abs{K_x}^{1/p}\norm{f(s)-f(t)}_{L^\infty_{x,v}}
 \longrightarrow0\qquad(s\to t),
\]
where we denote the measure of a measurable set $A$ by $\abs{A}$.
Consequently, we have that 
\begin{equation}\label{eq:time-sup}
 \norm{f}_{L^\infty([0,T];X_p)}
 =\sup_{0\leq s\leq T}\norm{f(s)}_{X_p}.
\end{equation}

We prove that the trajectory solution map $\mathcal S_T$ is discontinuous on every
time interval $[0,T]$, with observation times tending to zero. In what follows, we write $B_R := \{v\in\R^d:\abs{v}<R\}$.
\begin{theorem}
\label{thm:norm-inflation}
Let $d \le 3$, $\sigma\in\{-1,1\}$, and $1\leq p<\infty$.
There exist a compact set $K_x\subset\R^d$, a radius $R>0$, a sequence
$t_n\downarrow0$, and data $f_{0,n}\in\D$ such that
\begin{equation}\label{eq:theorem-data}
 0\leq f_{0,n}\leq1,\qquad
 \operatorname{supp}f_{0,n}\subset K_x\times B_R,\qquad
 \norm{f_{0,n}}_{X_p}\to0,\qquad
 \norm{f_{0,n}}_{L^1_{x,v}}\to0,
\end{equation}
while the associated classical solutions satisfy
\begin{equation}\label{eq:theorem-inflation}
 \norm{f_n(t_n)}_{X_p}\geq1.
\end{equation}
In particular, for every $T>0$, the trajectory map $\mathcal S_T$ is
discontinuous at zero.
\end{theorem}

The fixed-time solution map $S_t$ is also discontinuous at every
sufficiently small positive time.

\begin{theorem}
\label{thm:fixed-time}
Let $d \le 3$ and $\sigma\in\{-1,1\}$.
There exist $t_*>0$, $R>0$, and a compact set $K_x\subset\R^d$, independent
of $p$, $t$, and $n$, with the following property.  For every
$1\leq p<\infty$ and $t\in(0,t_*)$ there are data
$f_{0,n}^{(t)}\in\D$ satisfying
\begin{equation}\label{eq:fixed-data}
 \begin{gathered}
 0\leq f_{0,n}^{(t)}\leq1,\qquad
 \operatorname{supp}f_{0,n}^{(t)}\subset K_x\times B_R,\qquad 
 \norm{f_{0,n}^{(t)}}_{X_p}\longrightarrow0,\qquad
 \norm{f_{0,n}^{(t)}}_{L^1_{x,v}}\longrightarrow0,
 \end{gathered}
\end{equation}
but
\begin{equation}\label{eq:fixed-conclusion}
 \norm{f_n^{(t)}(t)}_{X_p}\geq1
\end{equation}
for every $n$.  Hence $S_t:(\D,\norm{\cdot}_{X_p})\to X_p$ is
discontinuous at zero.
\end{theorem}

\begin{remark}
	The condition $d \le 3$ enter only through the global classical
	existence theory for smooth compactly supported data. In higher dimensions, the same argument applies whenever the
	classical solutions exist up to their observation times.
\end{remark}

The construction exploits the order of the mixed norm, which permits a
different velocity to realize the supremum at each spatial point.  We
place packets of radius $r$ on a spatial lattice of spacing $h$ and use a
fixed velocity cutoff.  Their initial spatial envelope has measure
$O((r/h)^d)$, so its $X_p$ norm is $O((r/h)^{d/p})$ and its mass is
$O((r/h)^d)$.  At the observation time $h$, each $x$ in a fixed unit cube
has a nearest lattice point $z$ and an admissible velocity
$v=(x-z)/h$; the free characteristic through $(x,v)$ therefore starts at
$z$, making the transported velocity envelope equal to one throughout
the cube.  For the nonlinear Vlasov--Poisson flow, the Coulomb-field
interpolation estimate and a velocity-support bootstrap give
\begin{align*}
 M^{1/d}&\leq C\frac rh,
 &\sup_{0\leq s\leq h}\norm{E_f(s)}_{L^\infty}&\leq C\frac rh,\\
 \abs{V(0;h,x,v)-v}&\leq Cr,
 &\abs{X(0;h,x,v)-(x-hv)}&\leq Chr.
\end{align*}
Thus, for sufficiently small $h$, the nonlinear backward characteristic
still lands inside the respective spatial and velocity plateaux,
uniformly as $r\to0$.  Taking $h_n\downarrow0$ and $r_n=h_n^3$ gives the
vanishing-time result, whereas fixing $h=t$ and sending $r_n\to0$ gives
the fixed-time result.

Several lower-regularity theories provide useful points of comparison.
Jeong and Tae~\cite{JeongTae} and Tae~\cite{Tae} prove local
well-posedness in phase-space Sobolev classes, while Loeper~\cite{Loeper}
and Iacobelli and Junn\'e~\cite{IacobelliJunne} obtain stability in
transportation distances.  More directly related to the mixed norm,
Nguyen~\cite[Theorem~1]{NguyenFractionalVP} assumes, for $p>d$ and
$m>d+1$, $L_x^p$ control of a velocity supremum and a weighted uniform
$C_v^\theta$ modulus.  On sets where the mass and these two envelopes
are uniformly bounded, the solution map is continuous from
$L^1_{x,v}\cap L^p_{x,v}$ into
$C([0,T];L^1_{x,v}\cap L^p_{x,v})$.  Our packet data satisfy vanishing
bounds for these initial quantities, as recorded in
Section~\ref{sec:consequences}; the discontinuity proved here concerns
the different target topology $X_p$.

Critical Sobolev ill-posedness for Euler and SQG can produce arbitrary
norm amplification or nonexistence in the critical solution class; see
Bourgain and Li~\cite{BourgainLi2015,BourgainLi2021}, Kim and
Jeong~\cite{KimJeongEuler}, and Jeong and Kim~\cite{JeongKimSQG}.  Those
arguments exploit deformation or stretching of transported quantities.
Here the obstruction is already present in free kinetic transport, and
the output $X_p$ norms remain uniformly bounded.  Our conclusion is
therefore discontinuity of the solution maps, not arbitrary norm
inflation or nonexistence for rough data.

The sequences in both theorems converge to zero in the stronger
initial-data norm
$\norm{\cdot}_{X_p}+\norm{\cdot}_{L^1_{x,v}}$.  Thus the same
discontinuity statements hold when $\D$ carries this topology.  Since
$\mathcal S_T(0)=0$ and $S_t(0)=0$, respectively, no extension of either
map to the $X_p$-closure of $\D$ can be continuous at zero while agreeing
with the corresponding classical map on $\D$.

Section~\ref{sec:estimates} establishes the field and characteristic
estimates; Section~\ref{sec:packets} combines them with the packet
construction to prove the theorems.  Section~\ref{sec:discussion}
records the main obstruction to extending the argument to other kinetic
models.

\section{Field and characteristic estimates}\label{sec:estimates}

The Coulomb field satisfies the following interpolation estimate.
Together with conservation of mass and the supremum norm, it controls
the characteristics without bounds on derivatives of the data.

\begin{lemma}
\label{lem:field}
Let $d\geq1$ and let $\rho\in L^1(\R^d)\cap L^\infty(\R^d)$.  Then
\begin{equation}\label{eq:field-bound}
 \norm{K_d*\rho}_{L^\infty}
 \leq C_d
 \norm{\rho}_{L^1}^{1/d}
 \norm{\rho}_{L^\infty}^{1-1/d}.
\end{equation}
\end{lemma}

\begin{proof}
Fix $x\in\R^d$ and a radius $s>0$.  Since
$\abs{K_d(z)}\leq C_d\abs{z}^{1-d}$, splitting the convolution into
$\abs{x-y}\leq s$ and $\abs{x-y}>s$ gives
\begin{align}
 \abs{(K_d*\rho)(x)}
 &\leq C_d\norm{\rho}_{L^\infty}
       \int_{\abs{z}\leq s}\abs{z}^{1-d}\dd z
   +C_d s^{1-d}\norm{\rho}_{L^1} \notag\\
 &\leq C_d\left(
       s\norm{\rho}_{L^\infty}
       +s^{1-d}\norm{\rho}_{L^1}\right).
 \label{eq:field-split}
\end{align}
If $\rho=0$, there is nothing to prove.  Otherwise choose
\begin{equation*}
 s=\left(\frac{\norm{\rho}_{L^1}}
                 {\norm{\rho}_{L^\infty}}\right)^{1/d}.
\end{equation*}
Substitution in \eqref{eq:field-split} gives \eqref{eq:field-bound}.
\end{proof}

Mass and initial velocity support control the deviation of the
characteristics from free transport.

\begin{lemma}
\label{lem:characteristics}
Let $T>0$ and let $f$ be the classical solution of \eqref{eq:VP}
with nonnegative initial datum $f_0\in C_c^\infty(\R^{2d})$.
Let $R_0>0$, and suppose
\begin{equation}\label{eq:char-assumptions}
 0\leq f_0\leq1,
 \qquad
 \operatorname{supp}_v f_0\subset B_{R_0},
 \qquad
 M:=\iint_{\R^{2d}}f_0\dd x\dd v<\infty.
\end{equation}
There are constants $c_0=c_0(d,R_0)>0$ and
$C_0=C_0(d,R_0)>0$ such that, if
\begin{equation}\label{eq:smallness-TM}
 T M^{1/d}\leq c_0,
\end{equation}
then, for $0\leq s\leq T$,
\begin{equation}\label{eq:P-E-bounds}
 \operatorname{supp}_v f(s)\subset B_{R_0+1},
 \qquad
 \norm{E_f(s)}_{L^\infty}\leq C_0M^{1/d}.
\end{equation}
Moreover, every backward characteristic ending at $(x,v)$ at time
$t\leq T$ satisfies
\begin{align}
 \abs{V(s;t,x,v)-v}
 &\leq C_0(t-s)M^{1/d},
 &&0\leq s\leq t, \label{eq:V-error}\\
 \abs{X(0;t,x,v)-(x-tv)}
 &\leq \frac{C_0}{2}t^2M^{1/d}.
 \label{eq:X-error}
\end{align}
\end{lemma}

\begin{proof}
The transport identity \eqref{eq:transport} and incompressibility of the
phase-space vector field imply
\begin{equation}\label{eq:conservation}
 0\leq f(s)\leq1,
 \qquad
 \iint f(s,x,v)\dd x\dd v=M
\end{equation}
for every time for which the classical solution exists.
If $M=0$, nonnegativity and \eqref{eq:conservation} imply
$f\equiv0$, so the conclusions hold.  We may therefore assume $M>0$.

Define the nondecreasing velocity-support function
\begin{equation}\label{eq:Pdef}
 P(s)=\sup\left\{\abs{v}:
       f(\tau,x,v)\neq0\text{ for some }0\leq\tau\leq s
       \text{ and }x\in\R^d\right\}.
\end{equation}
Let $(X(\tau;0,x_0,v_0),V(\tau;0,x_0,v_0))$ denote the forward
characteristic from $(x_0,v_0)$.  Since the flow is continuous and
$\operatorname{supp}f_0$ is compact,
\[
 P(s)=\max_{\substack{(x_0,v_0)\in\operatorname{supp}f_0\\0\leq\tau\leq s}}
       \abs{V(\tau;0,x_0,v_0)}.
\]
Thus $P$ is finite and continuous, by uniform continuity of $V$ on
the compact set of initial points and times.

By \eqref{eq:conservation},
\begin{equation}\label{eq:rho-P}
 \rho_f(s,x)
 =\int_{\abs{v}\leq P(s)}f(s,x,v)\dd v
 \leq \abs{B_1}P(s)^d.
\end{equation}
Mass conservation gives $\norm{\rho_f(s)}_{L^1}=M$.  Lemma~\ref{lem:field}
therefore yields
\begin{equation}\label{eq:E-P}
 \norm{E_f(s)}_{L^\infty}
 \leq C_dM^{1/d}P(s)^{d-1}.
\end{equation}

Set $R_1=R_0+1$, and let
\begin{equation*}
 T_1=\sup\left\{\tau\in[0,T]:P(s)\leq R_1
                    \text{ for every }0\leq s\leq\tau\right\}.
\end{equation*}
This number is positive because $P(0)\leq R_0$.  For $0\leq s\leq T_1$,
\eqref{eq:E-P} gives
\begin{equation}\label{eq:E-bootstrap}
 \norm{E_f(s)}_{L^\infty}
 \leq C_dR_1^{d-1}M^{1/d}=:C_0M^{1/d}.
\end{equation}
Along any characteristic issuing from the initial support,
\begin{equation*}
 \abs{V(s)}
 \leq R_0+\int_0^s\norm{E_f(\tau)}_{L^\infty}\dd\tau
 \leq R_0+C_0TM^{1/d}.
\end{equation*}
Choose $c_0>0$ so that $C_0c_0\leq1/2$.  Under
\eqref{eq:smallness-TM}, we then have $P(s)\leq R_0+1/2<R_1$ on
$[0,T_1]$.  By continuity, $T_1<T$ would contradict the definition of
$T_1$.  Hence $T_1=T$, proving \eqref{eq:P-E-bounds}.

For a backward characteristic with $(X(t),V(t))=(x,v)$, integration of the
second equation in \eqref{eq:characteristics} gives
\begin{equation*}
 V(s)-v=-\sigma\int_s^tE_f(\tau,X(\tau))\dd\tau.
\end{equation*}
The field bound in \eqref{eq:P-E-bounds} proves \eqref{eq:V-error}.  Since
\begin{equation*}
 X(0)=x-\int_0^tV(s)\dd s,
\end{equation*}
we also obtain
\begin{align*}
 \abs{X(0)-(x-tv)}
 &\leq\int_0^t\abs{V(s)-v}\dd s\\
 &\leq C_0M^{1/d}\int_0^t(t-s)\dd s
 =\frac{C_0}{2}t^2M^{1/d},
\end{align*}
which is \eqref{eq:X-error}.
\end{proof}

\section{Packet construction and proofs of the main results}\label{sec:packets}

\subsection{The packet construction}

We place the packets on a lattice around the unit cube
\begin{equation}\label{eq:cubes}
 Q=(0,1)^d.
\end{equation}
Choose functions $\eta,\psi\in C_c^\infty(\R^d)$ such that
\begin{align}
 &0\leq\eta\leq1,
 \qquad
 \eta=1\text{ on }B_{1/2},
 \qquad
 \operatorname{supp}\eta\subset B_1, \label{eq:eta}\\
 &0\leq\psi\leq1,
 \qquad
 \psi=1\text{ on }B_{2\sqrt d},
 \qquad
 \operatorname{supp}\psi\subset B_R                 \label{eq:psi}
\end{align}
for some fixed $R>2\sqrt d$.

Given parameters $0<h<1$ and $0<r<h/4$, define the finite grid
\begin{equation}\label{eq:grid}
 \Gamma_h=h\mathbb Z^d\cap[-1,2]^d
\end{equation}
and the spatial packet profile
\begin{equation}\label{eq:a-hr}
 a_{h,r}(x)=\sum_{z\in\Gamma_h}
             \eta\left(\frac{x-z}{r}\right).
\end{equation}
The distance between two different grid points is at least $h$, so the
supports in \eqref{eq:a-hr} are disjoint.  We set
\begin{equation}\label{eq:f-hr}
 f_0^{h,r}(x,v)=a_{h,r}(x)\psi(v).
\end{equation}

Here $h$ is both the lattice spacing and the observation time.  Since
$\#\Gamma_h=O(h^{-d})$, the packets occupy spatial measure
$O((r/h)^d)$; hence $r/h\to0$ makes the initial $X_p$ norm and mass
vanish.  Lemma~\ref{lem:characteristics} gives velocity and position
errors $O(r)$ and $O(hr)$ on $[0,h]$.  Thus small $h$ keeps the spatial
error inside a packet of radius $r$.  The choice $r=h^3$ in the proof of
Theorem~\ref{thm:norm-inflation} is convenient but not essential; in
particular, the characteristic argument does not require $r\ll h^2$.

\begin{proposition}
\label{prop:packets}
For $d \le 3$ and either $\sigma\in\{-1,1\}$, there exist
$h_*\in(0,1]$ and $r_*>0$, depending only on $d$, $\eta$,
and $\psi$, such that for every
\[
 0<h<h_*,
 \qquad
 0<r<\min\{h/4,r_*\},
\]
the datum \eqref{eq:f-hr} has the following properties for every
$1\leq p<\infty$:
\begin{align}
 &f_0^{h,r}\in C_c^\infty(\R^{2d}),
 \qquad
 0\leq f_0^{h,r}\leq1,
 \qquad
 \operatorname{supp}_vf_0^{h,r}\subset B_R,
 \label{eq:packet-basic}\\
 &\norm{f_0^{h,r}}_{X_p}
 \leq C_{d,p,\eta}\left(\frac rh\right)^{d/p},
 \label{eq:packet-Xp}\\
 &M_{h,r}:=\norm{f_0^{h,r}}_{L^1_{x,v}}
 \leq C_{d,\eta,\psi}\left(\frac rh\right)^d.
 \label{eq:packet-mass}
\end{align}
If $f^{h,r}$ is the corresponding classical solution, then
\begin{equation}\label{eq:packet-output-pointwise}
 \mathop{\rm ess\,sup}_{v\in\R^d}f^{h,r}(h,x,v)=1
 \qquad\text{for every }x\in Q.
\end{equation}
In particular,
\begin{equation}\label{eq:packet-output}
 \norm{f^{h,r}(h)}_{X_p}\geq\abs{Q}^{1/p}=1.
\end{equation}
\end{proposition}

\begin{proof}
The number of grid points satisfies
\[
 \#\Gamma_h\leq C_dh^{-d}.
\]
Because the supports of the terms in \eqref{eq:a-hr} are disjoint,
$0\leq a_{h,r}\leq1$, which gives \eqref{eq:packet-basic}.
Since $h<1$ and $r<h/4<1/4$, we also have the common spatial support bound
\begin{equation}\label{eq:common-support}
 \operatorname{supp}a_{h,r}\subset K_x:=[-2,3]^d.
\end{equation}
Moreover,
\begin{align}
 \norm{f_0^{h,r}}_{X_p}^p
 &=\norm{a_{h,r}}_{L^p_x}^p \notag\\
 &=(\#\Gamma_h)r^d\norm{\eta}_{L^p}^p
 \leq C_{d,\eta}h^{-d}r^d.
 \label{eq:packet-Xp-proof}
\end{align}
Taking the $p$th root proves \eqref{eq:packet-Xp}.  Similarly,
\begin{align}
 M_{h,r}
 &=\left(\int_{\R^d}a_{h,r}(x)\dd x\right)
   \left(\int_{\R^d}\psi(v)\dd v\right)\notag\\
 &=(\#\Gamma_h)r^d\norm{\eta}_{L^1}\norm{\psi}_{L^1}
 \leq C_{d,\eta,\psi}\left(\frac rh\right)^d,
 \label{eq:packet-mass-proof}
\end{align}
which proves \eqref{eq:packet-mass}.

To estimate the displacement caused by the field, note first that
\eqref{eq:packet-mass} gives
\begin{equation}\label{eq:M-root}
 M_{h,r}^{1/d}\leq C_1\frac rh.
\end{equation}
Take $T=h$ in Lemma~\ref{lem:characteristics}.  Its smallness condition is
satisfied once
\begin{equation*}
 hM_{h,r}^{1/d}\leq C_1r\leq c_0.
\end{equation*}
Write $E^{h,r}:=E_{f^{h,r}}$.  Under this condition, for $0\leq s\leq h$,
\begin{equation}\label{eq:packet-field}
 \norm{E^{h,r}(s)}_{L^\infty}\leq C_2\frac rh.
\end{equation}
For every backward characteristic ending at $(x,v)$ at time $h$,
Lemma~\ref{lem:characteristics} and \eqref{eq:M-root} give
\begin{align}
 \abs{V(0;h,x,v)-v}
 &\leq C_3hM_{h,r}^{1/d}\leq C_3r,
 \label{eq:packet-V-error}\\
 \abs{X(0;h,x,v)-(x-hv)}
 &\leq C_3h^2M_{h,r}^{1/d}\leq C_3hr.
 \label{eq:packet-X-error}
\end{align}
Here $C_3=C_0C_1$ is admissible, with $C_0$ as in
Lemma~\ref{lem:characteristics} for $R_0=R$.

Fix $x\in Q$.  For $h<1$, a nearest point of the full lattice
$h\mathbb Z^d$ to $x$ belongs to $[-1,2]^d$.  Hence there is a point
$z=z_h(x)\in\Gamma_h$ such that
\begin{equation}\label{eq:nearest-grid}
 \abs{x-z}\leq\frac{\sqrt d}{2}h.
\end{equation}
Define the final velocity
\begin{equation}\label{eq:target-velocity}
 v_h(x)=\frac{x-z_h(x)}h.
\end{equation}
Then
\begin{equation}\label{eq:target-velocity-bound}
 \abs{v_h(x)}\leq\frac{\sqrt d}{2}.
\end{equation}
For free transport, the backward trajectory with terminal point
$(x,v_h(x))$ lands exactly at
\[
 (x-hv_h(x),v_h(x))=(z_h(x),v_h(x)).
\]
The spatial component is the center of the $z_h(x)$ packet, and the
velocity component lies strictly inside the plateau of $\psi$.  We now
verify that the nonlinear characteristic stays within these two
plateaux.

Let $(X(s),V(s))$ be the backward characteristic satisfying
$(X(h),V(h))=(x,v_h(x))$.  Set
\[
 h_*=\min\left\{1,\frac{1}{4C_3}\right\},
 \qquad
 r_*=\min\left\{\frac{c_0}{2C_1},\frac{\sqrt d}{2C_3}\right\}.
\]
These choices give
\begin{equation}\label{eq:hstar-choice}
 C_3h_*<\frac12
\end{equation}
and
\begin{equation}\label{eq:rstar-choice}
 C_1r_*\leq c_0,
 \qquad
 C_3r_*<\sqrt d.
\end{equation}
Because $x-hv_h(x)=z_h(x)$, the spatial error estimate and
\eqref{eq:hstar-choice} give
\begin{equation}\label{eq:initial-X-plateau}
 \abs{X(0)-z_h(x)}<\frac r2,
 \qquad
 a_{h,r}(X(0))=1.
\end{equation}
Indeed, the first inequality places $X(0)$ in the ball on which the
$z_h(x)$ summand in \eqref{eq:a-hr} equals one; disjointness and
$0\leq a_{h,r}\leq1$ then give the second equality.  For the velocity
component, \eqref{eq:target-velocity-bound},
\eqref{eq:packet-V-error}, and \eqref{eq:rstar-choice} yield the explicit
margin
\[
 \abs{V(0)}
 \leq \abs{v_h(x)}+C_3r
 <\frac{\sqrt d}{2}+\sqrt d
 =\frac{3\sqrt d}{2}<2\sqrt d.
\]
Hence
\[
 \psi(V(0))=1.
\]
The plateau estimates are uniform in $x\in Q$ and $r<r_*$.  Conservation
along characteristics therefore gives
\begin{equation}\label{eq:value-one}
 f^{h,r}(h,x,v_h(x))
 =f_0^{h,r}(X(0),V(0))=1.
\end{equation}

The solution is continuous and satisfies $0\leq f^{h,r}\leq1$.
For each $0<\varepsilon<1$, \eqref{eq:value-one} and continuity in
velocity give an open ball on which $f^{h,r}(h,x,v)>1-\varepsilon$.
Since this ball has positive measure,
\begin{equation*}
 \mathop{\rm ess\,sup}_{v\in\R^d}f^{h,r}(h,x,v)=1.
\end{equation*}
This is \eqref{eq:packet-output-pointwise}.  Integrating over $Q$ gives
\eqref{eq:packet-output}.
\end{proof}

\subsection{Proofs of the main theorems}\label{sec:proofs}

Both theorems follow by choosing the packet spacing and radius in
Proposition~\ref{prop:packets}.

\begin{proof}[Proof of Theorem~\ref{thm:norm-inflation}]
Choose a sequence $h_n\downarrow0$ such that, for every $n$,
\[
 h_n<h_*,
 \qquad
 h_n^3<\min\{h_n/4,r_*\},
\]
and set
\begin{equation}\label{eq:diagonal-choice}
 t_n=h_n,
 \qquad
 r_n=h_n^3.
\end{equation}
Define
\begin{equation*}
 f_{0,n}=f_0^{h_n,r_n}.
\end{equation*}
The amplitude and common support assertions in \eqref{eq:theorem-data}
follow from \eqref{eq:packet-basic} and \eqref{eq:common-support}.  Since $r_n/h_n=h_n^2$,
\eqref{eq:packet-Xp} and \eqref{eq:packet-mass} give
\begin{align*}
 \norm{f_{0,n}}_{X_p}
 &\leq C h_n^{2d/p}\longrightarrow0,\\
 \norm{f_{0,n}}_{L^1_{x,v}}
 &\leq C h_n^{2d}\longrightarrow0.
\end{align*}
Equation~\eqref{eq:packet-output} gives
\begin{equation*}
 \norm{f_n(t_n)}_{X_p}\geq1.
\end{equation*}
If $\mathcal S_T$ were continuous at zero with values in
$L^\infty([0,T];X_p)$, then \eqref{eq:time-sup} and the initial convergence in $X_p$ would imply
\begin{equation*}
 \sup_{0\leq t\leq T}\norm{f_n(t)}_{X_p}\longrightarrow0.
\end{equation*}
For all large $n$, however, $t_n<T$ and the left side is at least one,
a contradiction.
\end{proof}

\begin{proof}[Proof of Theorem~\ref{thm:fixed-time}]
Take $t_*=h_*$ and the sets $K_x$ and $B_R$ from the construction.
Fix $t\in(0,t_*)$, set $h=t$, and choose any sequence
$r_n\downarrow0$ such that
$r_n<\min\{t/4,r_*\}$ for every $n$.  Put
\begin{equation*}
 f_{0,n}^{(t)}=f_0^{t,r_n}.
\end{equation*}
Proposition~\ref{prop:packets} and \eqref{eq:common-support} give the
amplitude and common support bounds.  Since $t$ is fixed and $r_n/t\to0$,
\eqref{eq:packet-Xp} and \eqref{eq:packet-mass} give
\begin{equation*}
 \norm{f_{0,n}^{(t)}}_{X_p}
 \leq C\left(\frac{r_n}{t}\right)^{d/p}\longrightarrow0,
 \qquad
 \norm{f_{0,n}^{(t)}}_{L^1_{x,v}}
 \leq C\left(\frac{r_n}{t}\right)^d\longrightarrow0.
\end{equation*}
On the other hand, \eqref{eq:packet-output} yields
\begin{equation*}
 \norm{f_n^{(t)}(t)}_{X_p}\geq1
\end{equation*}
for every $n$.  Thus $S_t$ is discontinuous at zero.
\end{proof}

\subsection{Consequences and endpoint remarks}\label{sec:consequences}

The same construction gives a self-contained free-transport statement.

\begin{corollary}[Free transport]
\label{cor:free-transport}
Let $d \le 3$ and $1\leq p<\infty$.  For every $h\in(0,1)$, the
fixed-time free-transport map
\[
 U(h):(\D,\norm{\cdot}_{X_p})\longrightarrow X_p,
 \qquad
 U(h)g(x,v)=g(x-hv,v),
\]
is discontinuous at zero.  Moreover, for every $T>0$, the trajectory map
\[
 \mathcal U_T:(\D,\norm{\cdot}_{X_p})
 \longrightarrow L^\infty([0,T];X_p),
 \qquad
 \mathcal U_Tg=U(\cdot)g,
\]
is discontinuous at zero.
In both assertions, the approximating data may be chosen nonnegative,
bounded by one, supported in a common compact subset of phase space, and
convergent to zero in $L^1_{x,v}$.
\end{corollary}

\begin{proof}
Fix $h\in(0,1)$ and choose $r_n\downarrow0$ with $r_n<h/4$.  Set
$g_n=f_0^{h,r_n}$.  By construction, $0\leq g_n\leq1$ and
$\operatorname{supp}g_n\subset[-2,3]^d\times B_R$.  The computations in
\eqref{eq:packet-Xp-proof} and \eqref{eq:packet-mass-proof} give
\[
 \norm{g_n}_{X_p}
 \leq C\left(\frac{r_n}{h}\right)^{d/p}\longrightarrow0,
 \qquad
 \norm{g_n}_{L^1_{x,v}}
 \leq C\left(\frac{r_n}{h}\right)^d\longrightarrow0.
\]
For each $x\in Q$, choose $z_h(x)$ and $v_h(x)$ as in
\eqref{eq:nearest-grid}--\eqref{eq:target-velocity}.  Then
$x-hv_h(x)=z_h(x)$ and $\abs{v_h(x)}\leq\sqrt d/2$.  Since
$\eta(0)=1$ and $\psi=1$ on $B_{2\sqrt d}$,
\[
 U(h)g_n(x,v_h(x))
 =a_{h,r_n}(z_h(x))\psi(v_h(x))=1.
\]
The function $v\mapsto U(h)g_n(x,v)$ is continuous and takes values in
$[0,1]$.  Hence its essential supremum equals one for every $x\in Q$,
and therefore
\[
 \norm{U(h)g_n}_{X_p}\geq\abs{Q}^{1/p}=1.
\]
For the trajectory assertion, choose $h_n\downarrow0$ and
$r_n=h_n^3$ so that
\[
 h_n<\min\{T,1\},
 \qquad
 r_n<h_n/4
\]
for every $n$, and set $g_n=f_0^{h_n,r_n}$.  Then
\[
 \norm{g_n}_{X_p}\leq Ch_n^{2d/p}\longrightarrow0,
 \qquad
 \norm{g_n}_{L^1_{x,v}}\leq Ch_n^{2d}\longrightarrow0,
 \qquad
 \sup_{0\leq t\leq T}\norm{U(t)g_n}_{X_p}
 \geq\norm{U(h_n)g_n}_{X_p}\geq1.
\]
The same amplitude and common-support bounds hold for this sequence.
Thus $\mathcal U_T$ is discontinuous at zero.
\end{proof}

The proof uses a terminal velocity $v_h(x)$ that varies with $x$, so
taking the velocity supremum before integrating in space is essential.
By contrast, the reversed norm
\begin{equation*}
 L_v^\infty L_x^p,
 \qquad
 \norm{g}_{L_v^\infty L_x^p}
 =\mathop{\rm ess\,sup}_v\norm{g(\cdot,v)}_{L_x^p},
\end{equation*}
is preserved by free transport because, for each fixed $v$, the map
$x\mapsto x-tv$ is a translation:
\[
 \norm{U(t)g}_{L_v^\infty L_x^p}
 =\mathop{\rm ess\,sup}_v
   \norm{g(\cdot-tv,v)}_{L_x^p}
 =\norm{g}_{L_v^\infty L_x^p}.
\]

At $p=\infty$, the mixed norm is the phase-space supremum norm and is
preserved by the classical flow.  Since $S_t(0)=0$, we have
\[
 \norm{S_tg-S_t(0)}_{L^\infty_{x,v}}=\norm{g}_{L^\infty_{x,v}},
 \qquad
 \norm{\mathcal S_Tg}_{L^\infty([0,T];L^\infty_{x,v})}
 =\norm{g}_{L^\infty_{x,v}}.
\]
Thus both solution maps are continuous at zero in the supremum-norm
topology, so the discontinuity conclusions do not extend to this endpoint.
For every finite $q\geq1$, the
constructed data satisfy
\[
 \norm{f_0}_{L^q_{x,v}}^q\leq\norm{f_0}_{L^1_{x,v}}=M.
\]
The same inequality holds at every later time by transport and mass
conservation.  Hence the solutions in either construction converge to
zero in every finite phase-space $L^q$ norm, uniformly in time, even
though their $X_p$ norms at the observation times are at least one.

For the comparison with \cite{NguyenFractionalVP}, define the weighted
velocity envelopes
\begin{align*}
 G_0(x)
 &=\mathop{\rm ess\,sup}_{v\in\R^d}
   \langle v\rangle^m\abs{f_0(x,v)},\\
 G_\theta(x)
 &=\mathop{\rm ess\,sup}_{\substack{v\in\R^d\\0<\abs{w}\leq1}}
   \langle v\rangle^m
   \frac{\abs{f_0(x,v+w)-f_0(x,v)}}{\abs{w}^\theta}.
\end{align*}
For $f_0^{h,r}=a_{h,r}\psi$, the fixed smooth compactly supported
velocity cutoff gives
\[
 G_0(x)+G_\theta(x)\leq C_{m,\theta,\psi}a_{h,r}(x).
\]
Consequently, for every fixed $p>d$, $m>d+1$, and $0<\theta\leq1$,
\[
 \norm{f_0^{h,r}}_{L^1_{x,v}}
 +\norm{G_0}_{L_x^p}+\norm{G_\theta}_{L_x^p}
 \longrightarrow0
 \qquad\text{as }r/h\longrightarrow0.
\]
Thus the present discontinuity in $X_p$ is compatible with the
phase-space continuous dependence in \cite{NguyenFractionalVP}.

The output norms also have a uniform upper bound:
\[
 \norm{f^{h,r}(h)}_{X_p}
 \leq \bigl[3+2r+2(R+1)h\bigr]^{d/p}.
\]
Indeed, the initial spatial support lies in $[-1-r,2+r]^d$ and the speed
on the transported support is at most $R+1$ up to time $h$.
Combining this spatial support bound with $0\leq f\leq1$ proves the
displayed estimate.  Hence the output norms remain uniformly bounded
despite the discontinuity at zero.

\section{Obstructions in other kinetic models}\label{sec:discussion}

The packet proof uses two properties of Vlasov--Poisson.  First, the free
kinetic flow maps a sparse lattice of packets across a fixed spatial set.
Second, the self-consistent force perturbs the relevant characteristics
by less than the widths of the spatial and velocity plateaux.  The first
property persists in several kinetic models; the second is
model-dependent.

For the three-dimensional relativistic Vlasov--Maxwell system, the free
flow has spatial velocity
\[
 \widehat v=\frac{v}{\sqrt{1+\abs{v}^2}}.
\]
If \(z\) is the nearest lattice point to \(x\), then
\(u=(x-z)/h\) satisfies \(\abs{u}\leq\sqrt3/2<1\), and the momentum
\(v=u/\sqrt{1-\abs{u}^2}\) gives \(x-h\widehat v=z\).  Thus the free
geometry survives with a fixed momentum cutoff.  The present packet
argument would extend under a bound of the form
\[
 \sup_{0\leq s\leq h}
 \bigl(\norm{E(s)}_{L^\infty}+\norm{B(s)}_{L^\infty}\bigr)
 \leq C\frac rh .
\]
Small particle mass controls the longitudinal Coulomb field but not the
evolving transverse fields.  The momentum-support continuation criterion
of Glassey and Strauss~\cite{GlasseyStrauss} does not provide this bound,
so the Vlasov--Maxwell extension does not follow from the present
argument.

Collisions create a different obstruction.  For a nonnegative mild
solution of an angular-cutoff Boltzmann equation, write
\[
 Q(f,f)=Q^+(f,f)-\nu_f f.
\]
If the packet family has a common existence interval, $Q^+\geq0$, and
$\nu_f\leq C$ uniformly along the relevant free characteristics, the
gain--loss formula gives
\[
 f(h,x,v)\geq e^{-Ch}f_0(x-hv,v).
\]
The packet geometry is therefore compatible with cutoff collisions
under these hypotheses.  A discontinuity theorem still requires uniform
velocity-tail and loss-frequency estimates for the concentrated family;
we do not establish them here.

For the Landau equation and the Boltzmann equation without angular
cutoff, gain and loss cannot be separated in this way.  Moreover, along
the free profile \(a_{h,r}(x-sv)\psi(v)\), first and second velocity
derivatives can have sizes \(1+s/r\) and \(1+s^2/r^2\).  Formally, a
second-order collision term with bounded coefficients would therefore
produce an accumulated error of order
\[
 h+\frac{h^3}{r^2}.
\]
This scaling is only heuristic for the nonlinear collisional problem.
The required uniform collision estimate, particularly at a fixed
positive observation time as \(r\to0\), remains open.\vspace{0.5cm}

\noindent{\bf AI-use disclosure}. This manuscript was written by the authors. OpenAI’s ChatGPT was used
as an assistive tool for language editing, organization, bibliographic checks, and exploratory mathematical discussion.\vspace{0.3cm}

\noindent{\bf Acknowledgments}. The research of Ke Chen was supported by the Research Centre for Nonlinear Analysis at The Hong Kong Polytechnic University. The work of In-Jee Jeong was supported by the NRF grant from the Korea government (MSIT), No. 2022R1C1C1011051, RS-2024-00406821, and by a KIAS Individual Grant. The research of Quoc-Hung Nguyen was supported by the CAS Project for Young Scientists in Basic Research (Grant No. YSBR-031) and by the National Natural Science Foundation of China (Grant Nos. 1251101538 and 12595282).

\bigskip
{\small
\begin{flushleft}
\textsc{Ke Chen}\\
Department of Applied Mathematics, The Hong Kong Polytechnic University,\\
Kowloon, Hong Kong, P.~R.~China.\\
\emph{Email:} \href{mailto:k1chen@polyu.edu.hk}{\texttt{k1chen@polyu.edu.hk}}

\medskip
\textsc{In-Jee Jeong}\\
School of Mathematics, Korea Institute for Advanced Study,\\
85 Hoegi-ro, Seoul 02455, Republic of Korea.\\
\emph{Email:} \href{mailto:ijeong@kias.re.kr}{\texttt{ijeong@kias.re.kr}}

\medskip
\textsc{Quoc-Hung Nguyen}\\
Academy of Mathematics and Systems Science, Chinese Academy of Sciences,\\
Beijing 100190, P.~R.~China.\\
\emph{Email:} \href{mailto:qhnguyen@amss.ac.cn}{\texttt{qhnguyen@amss.ac.cn}}

\medskip
\textsc{Sangwook Tae}\\
Department of Mathematical Sciences and RIM, Seoul National University,\\
1 Gwanak-ro, Gwanak-gu, Seoul 08826, Republic of Korea.\\
\emph{Email:} \href{mailto:swtae00@snu.ac.kr}{\texttt{swtae00@snu.ac.kr}}
\end{flushleft}
}

\end{document}